\documentclass[oneside,a4paper]{amsart}

\usepackage[english]{babel}
\usepackage{amssymb,amsmath,amsthm,amsfonts}
\usepackage{txfonts,bbm}
\usepackage{graphicx}
\usepackage{array}

\usepackage{enumerate, hyphenat}
\usepackage{natbib}

\usepackage[usenames,dvipsnames]{xcolor}
\usepackage[colorlinks,linkcolor=MidnightBlue,citecolor=OliveGreen]{hyperref}

\usepackage[capitalise]{cleveref}

\newtheorem{theorem}{Theorem}[section]
\newtheorem{lemma}[theorem]{Lemma}
\newtheorem{proposition}[theorem]{Proposition}
\newtheorem{corollary}[theorem]{Corollary}

\theoremstyle{remark}
\newtheorem{remark}[theorem]{Remark}

\newcommand{\convd}{\stackrel{\mathcal{D}}{\to}}

\begin{document}

\title[First\hyp{}passage percolation on width\hyp{}two stretches]{First\hyp{}passage percolation on width\hyp{}two stretches with exponential link weights}
\author{Eckhard Schlemm}

\address{Wolfson College, Cambridge University}
\email{es555@cam.ac.uk}


\begin{abstract}
We consider the first\hyp{}passage percolation problem on effectively one\hyp{}dimensional graphs with vertex
set $\{1 . . . , n\}\times\{0, 1\}$ and translation\hyp{}invariant edge\hyp{}structure. For three of six non\hyp{}trivial cases we
obtain exact expressions for the asymptotic percolation rate $\chi$ by solving certain recursive distributional
equations and invoking results from ergodic theory to identify $\chi$ as the expected asymptotic one\hyp{}step
growth of the first\hyp{}passage time from $(0, 0)$ to $(n, 0)$.
\end{abstract}

\maketitle

\section{Introduction}
Let $G=(V,E)$ be a graph with vertex set $V=V(G)$ and unoriented edges $E=E(G)\subset V^2$. Two vertices $u,v\in V$ are said to be adjacent, for which we write $u\sim v$ if $(u,v)\in E$ and the edge $e=(u,v)$ is said to join the vertices $u$ and $v$. Assume there is a weight\hyp{}function $w:E\to\mathbb{R}$. For any two vertices $u,v\in V$ a {\it path $p$ joining $u$ and $v$ in $G$} is a collection of vertices $\{u=p_0,p_1,\dots,p_{n-1},p_n=v\}$ such that $p_\nu$ and $p_{\nu+1}$ are adjacent for $0\leq\nu<n$; a path $p$ is called {\it simple} if each vertex occurs in $p$ at most once. To a given path $p$ we associate the set of its comprising edges $\hat p = \{(p_\nu,p_{\nu+1}): 0\leq\nu<n\}$. The {\it weight} $w(p)$ of a path $p$ is then defined as $\sum_{e\in \hat p}{w(e)}$. We define $d_G: V\times V\to\mathbb{R}$ by $d_G(u,v)=\inf{\{w(p): p \text{ a path joining $u$ and $v$ in $G$}\}}$, and call $d_G(u,v)$ the {\it first\hyp{}passage time} between $u$ and $v$. A path p joining $u$ and $v$ in $G$, such that $w(p)=d_G(u,v)$ is called a {\it shortest path}. Throughout this work we assume that $G$ is finite and connected and we will be interested in the case that the weights $w(e)$, $e\in E$, are random variables; the goal is then to make probability statements about first\hyp{}passage times or related quantities. In general we note here that a shortest path $p=\{u=p_0,\dots,p_n=v\}$ is always simple and that each sub\hyp{}path $\{p_k,\dots,p_l\}$, $0\leq k < l\leq n$, is also a shortest path. Moreover, for continuously distributed, independent edge weights the shortest path between any two vertices is almost surely unique \citep{kesten1986aspects}.

The typical first\hyp{}passage percolation problem is based on two-dimensional regular infinite graphs $G$ with vertex sets $V(G)=\mathbb{Z}^2$. The edge weights $w(e)$, $e\in E(G)$, are i.i.d random variables in $L^1$ with some common distribution $\mathbb{P}$ such that $\mathbb{P}$-almost surely $w(e)$ is positive. Let $l_{m\to n}$, $0\leq m\leq n$, denote the first\hyp{}passage time from $(m,0)$ to $(n,0)$ subject to the condition that the contributing paths consist only of vertices with first coordinate $\nu$, $m\leq\nu\leq n$ and write $l_n\coloneqq l_{0\to n}$. By \cref{subadd}, $l_{0\to n}\leq l_{0\to m}+l_{m\to n}$ and the theory of sub\hyp{}additive processes (\cref{prop-FPP-rate}) implies that $\lim_{n\to\infty}{\frac{1}{n}l_n}$ exists and is almost surely constant. This number, which depends only on the graph $G$ and the distribution $\mathbb{P}$, is denoted by $\chi(G,\mathbb{P})$ and called the {\it (asymptotic) percolation rate} or {\it time constant}. The explicit calculation of the percolation rate even for the simplest regular infinite graphs and distributions $\mathbb{P}$ is characterized in \citet[p. 1937]{graham1995} as a "hopelessly intractable" problem.
\section{The model}
In this work we first focus on the first\hyp{}passage percolation problem as described in the last paragraph for certain families of regular, effectively one\hyp{}dimensional graphs with independent, random hyp{}weights. Let $V_n=\{0,1,\dots, n\}\times\{0,1\}$ and
\begin{align*}
\mathcal{V}_n =& \{\left((i,0),(i+1,1)\right):  0\leq i < n\} \\
\mathcal{W}_n =& \{\left((i,1),(i+1,0)\right):  0\leq i < n\}\\
\mathcal{X}_n =& \{\left((i,0),(i+1,0)\right):  0\leq i < n\} \\
\mathcal{Y}_n =& \{\left((i,1),(i+1,1)\right):  0\leq i < n\} \\
\mathcal{Z}_n =& \{\left((i,0),(i,1)\right):    0\leq i \leq n\}.
\end{align*}
To each subset $\mathcal{E}\subset\{\mathcal{V},\mathcal{W},\mathcal{X},\mathcal{Y},\mathcal{Z}\}$ corresponds a family of graphs
\begin{equation*}
\mathcal{G}^{\mathcal{E}}=\{G^{\mathcal{E}}_n\}_{n\in\mathbb{N}}=\{(V_n,\mathcal{E}_n)\}_{n\in\mathbb{N}}
\end{equation*}
with vertex sets $V_n$ and edges $\mathcal{E}_n=\bigcup_{\mathcal{L}\in \mathcal{E}}{\mathcal{L}_n}$. For each graph the edge weights are independent exponentially distributed random variables which are labelled $V_i,W_i,X_i,Y_i,Z_i$ in the obvious way. By time\hyp{}scaling it is no restriction of generality if we assume the parameter of the edge weight distributions to be unity. We denote the probability measure by $\mathbb{P}$, that is $d\mathbb{P}(x)=e^{-x}dx$. We point out that by construction, for any selection of edges $\mathcal{E}$, $G^{\mathcal{E}}_n$ is a subgraph of $G^\mathcal{E}_{n+1}$, an observation which forms the basis for our inductive argument described below. The method is based upon \citet{flaxman2006fpp} where it has successfully been employed to compute the asymptotic percolation rate on $\mathcal{G}^{\{\mathcal{X},\mathcal{Y},\mathcal{Z}\}}$ where the edge weights are taken to be independently one with probability $p$ and zero with probability $1-p$. They also consider continuous edge\hyp{}weight distributions and show that one can replace those by suitably chosen discrete ones to obtain arbitrarily good, yet approximative values for the time constant. We do not adopt this method here but rather work explicitly with the continuous distributions.
\section{Results}
It turns out that for only six families of graphs $\mathcal{G}^{\mathcal{E}}$ the first\hyp{}passage percolation is non\hyp{}trivial. For three of these six, namely for $\mathcal{G}^{\{\mathcal{X},\mathcal{Y},\mathcal{Z}\}}$, $\mathcal{G}^{\{\mathcal{V},\mathcal{W},\mathcal{X},\mathcal{Y}\}}$ and $\mathcal{G}^{\{\mathcal{W},\mathcal{X},\mathcal{Y},\mathcal{Z}\}}$ we derive expressions for $\chi$ which seem not to have been known so far (see \cref{tabl:graphs-2}). We are confident that our method works for the three remaining families, $\mathcal{G}^{\{\mathcal{V},\mathcal{W},\mathcal{X}\}}$, $\mathcal{G}^{\{\mathcal{V},\mathcal{W},\mathcal{X},\mathcal{Z}\}}$ and $\mathcal{G}^{\{\mathcal{V},\mathcal{W},\mathcal{X},\mathcal{Y},\mathcal{Z}\}}$, as well. Increased complexity in the calculations involved, however, prevents us from obtaining analytic results; the numerical values of $\chi$ for theses cases, which are also recorded in \cref{tabl:graphs-2}, were obtained by means of simulations.
\begin{table}
\caption[Width\hyp{}two stretches $G^E$ - Non\hyp{}trivial solved and unsolved cases]{Width\hyp{}two stretches $\mathcal{G}^{\mathcal{E}}$: Non\hyp{}trivial solved and unsolved cases}
\label{tabl:graphs-2}
\begin{tabular}{|m{0.21\textwidth}|m{.25\textwidth}|m{.3\textwidth}|}
\hline
\begin{minipage}[c]{0.21\textwidth}\begin{center}\boldmath{$ \mathcal{E}$}\end{center}\end{minipage} & \begin{minipage}[c]{.25\textwidth}\begin{center}\bfseries{Pictograph}\end{center}\end{minipage} & \begin{minipage}[b]{.3\textwidth}\begin{center}\boldmath{$ \chi\left(\mathcal{E}\right)$}\end{center}\end{minipage}\\
\hline\hline\hline
\begin{minipage}[c]{0.21\textwidth}\begin{center}$\{\mathcal{X},\mathcal{Y},\mathcal{Z}\}$\end{center}\end{minipage} &\begin{minipage}[c]{.25\textwidth}\includegraphics[width=\textwidth]{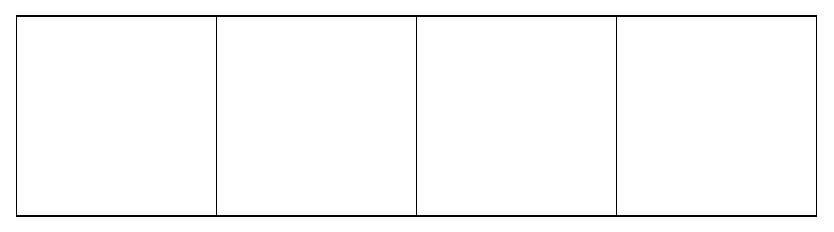}\end{minipage} & \begin{minipage}[b]{.3\textwidth}\begin{center}\[\frac{3}{2}-\frac{\operatorname{J}_1(2)}{2\operatorname{J}_2(2)}\]
\end{center}\end{minipage}\\
\hline
\begin{minipage}[c]{0.21\textwidth}\begin{center}$\{\mathcal{V},\mathcal{W},\mathcal{X},\mathcal{Y}\}$\end{center}\end{minipage} &\begin{minipage}[c]{.25\textwidth}\includegraphics[width=\textwidth]{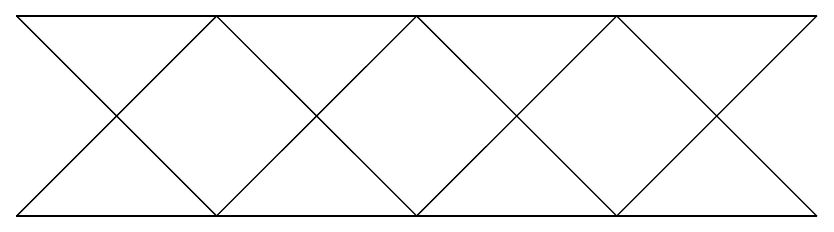}\end{minipage} & \begin{minipage}[b]{.3\textwidth}\begin{center}\[\frac{3}{4}-\frac{\operatorname{J}_0(\sqrt{2})}{2\sqrt{2}\operatorname{J}_1(\sqrt{2})}\] \end{center}\end{minipage}\\
\hline
\begin{minipage}[c]{0.21\textwidth}\begin{center}$\{\mathcal{W},\mathcal{X},\mathcal{Y},\mathcal{Z}\}$\end{center}\end{minipage} &\begin{minipage}[c]{.25\textwidth}\includegraphics[width=\textwidth]{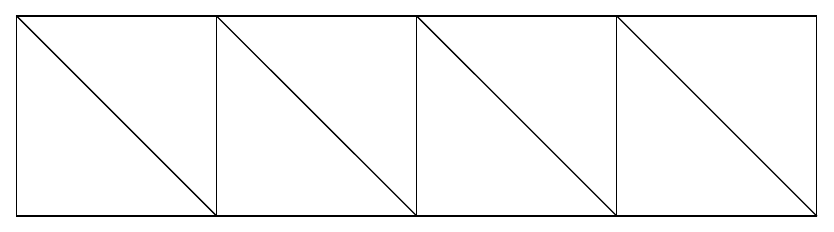}\end{minipage} & \begin{minipage}[b]{.3\textwidth}\begin{center}\[\frac{2\tan 1 - 2}{2\tan 1 - 1}\]\end{center}\end{minipage}\\
\hline
\hline
\begin{minipage}[c]{0.21\textwidth}\begin{center}$\{\mathcal{V},\mathcal{W},\mathcal{X}\}$\end{center}\end{minipage} & \begin{minipage}[c]{.25\textwidth}\includegraphics[width=\textwidth]{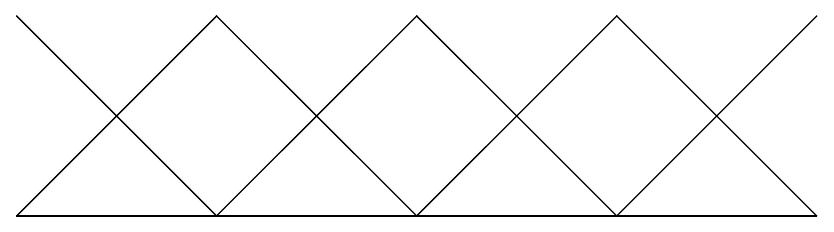}\end{minipage} & \begin{minipage}[b]{.3\textwidth}\begin{center}\[\approx.51\]\end{center}\end{minipage}\\
\hline
\begin{minipage}[c]{0.21\textwidth}\begin{center}$\{\mathcal{V},\mathcal{W},\mathcal{X},\mathcal{Z}\}$\end{center}\end{minipage} & \begin{minipage}[c]{.25\textwidth}\includegraphics[width=\textwidth]{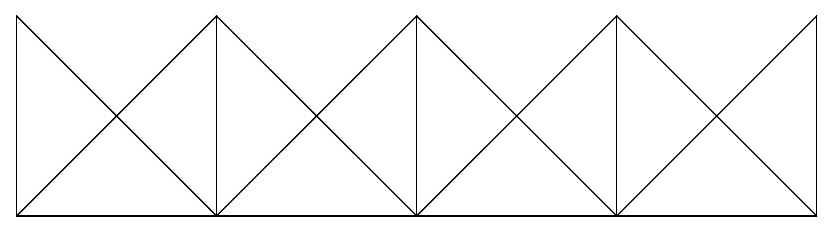}\end{minipage}& \begin{minipage}[b]{.3\textwidth}\begin{center}\[\approx .45\]\end{center}\end{minipage}\\
\hline
\begin{minipage}[c]{0.21\textwidth}\begin{center}$\{\mathcal{V},\mathcal{W},\mathcal{X},\mathcal{Y},\mathcal{Z}\}$\end{center}\end{minipage} & \begin{minipage}[c]{.25\textwidth}\includegraphics[width=\textwidth]{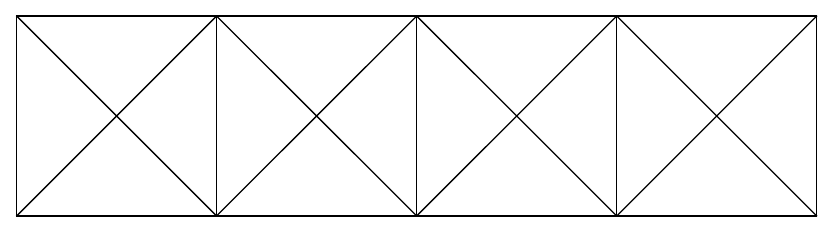}\end{minipage}& \begin{minipage}[b]{.3\textwidth}\begin{center}\[\approx .35\]\end{center}\end{minipage}\\
\hline
\end{tabular}
\end{table}

\section{Subadditivity and ergodic theorems}
In this section we present some theoretical results about first\hyp{}passage times. In particular we show that the asymptotic percolation rate $\lim_{n\to\infty}{\frac{1}{n}l_n}$ on graphs with vertex set $V_n\coloneqq\{0,\dots, n\}\times\{0,1\}$ and translation\hyp{}invariant edge structure is, with probability one, equal to a deterministic number and provide a formula for its computation (\cref{FPP-rate-formula}).
\begin{proposition}
\label{prop-FPP-rate}
Let $G$ be the graph $G_n^{\mathcal{E}}$ for some $\mathcal{E}\subset\{\mathcal{V},\mathcal{W},\mathcal{X},\mathcal{Y},\mathcal{Z}\}$. If the edge weights $w(e)$, $e\in \mathcal{E}_n$, are independent, identically $\mathbb{P}$-distributed, positive and integrable random variables then almost surely $\lim_{n\to\infty}{\frac{1}{n}l_n}$ exists and is equal to a deterministic number $\chi\left(G,\mathbb{P}\right)$.
\end{proposition}
To prove this we need a lemma about a property of first\hyp{}passage times, called {\it subadditivity}.
\begin{lemma}
\label{subadd}
For any weighted graph $G$, the function $d_G:V(G)^2\to\mathbb{R}$ is subadditive, that is $d_G(u,w)\leq d_G(u,v)+d_G(v,w)$ for all $u,v,w\in V(G)$. If in particular $V(G)=V_n$, this says $l_{0\to n}\leq l_{0\to m}+l_{m\to n} $ holds for any $0\leq m\leq n$.
\end{lemma}
\begin{proof}
Subadditivity of $d_G$ is a consequence of the simple observation that the set of all paths joining $u$ and $w$ in $G$ is a superset of the set of all paths joining $u$ and $w$ in $G$ and containing $v$ and that one can combine any two paths $p^{u\to v}$ joining $u$ and $v$ and $p^{v\to w}$ joining $v$ and $w$ to get a path $p^{u\to w}$ joining $u$ and $w$ that satisfies $w\left(p^{u\to w}\right) = w\left(p^{u\to v}\right)+w\left(p^{v\to w}\right)$. The remark about the special case $G=G_n^E$ is clear. (Take $u=(0,0)$, $v=(m,0)$, $w=(n,0)$)
\end{proof}
To conclude from this subadditivity property that the asymptotic percolation rate is almost surely a deterministic number we need the following {\it subadditive ergodic theorem} due to Liggett, who generalized a result of Kingman. (See \citet[theorem 6.1]{durrett1991pta} for a proof.)
\begin{theorem}[Liggett]
\label{liggett}
Suppose a family of random variables $X=\{X_{m,n}:0\leq m\leq n\}$ satisfies
\begin{enumerate}[(i)]
\item $X_{0,n}\leq X_{0,m}+X_{m.n}$.
\item For each $k\in\mathbb{N}$, $(X_{nk,(n+1)k})_{n\geq 0}$ is a stationary sequence.
\item The distribution of $(X_{m,m+k})_{k\geq 0}$ does not depend on $m$.
\item $\mathbb{E}\left[X_{0,1}^+\right]<\infty$ and for each $n\in\mathbb{N}$, $\mathbb{E}\left[X_{0,n}\right]\geq\gamma_0 $ holds with $\gamma_0>-\infty$.
\end{enumerate}
Then $X=\lim_{n\to\infty}{X_{0,n}/n}$ exists almost surely.
\end{theorem}
In the formulation of {\it (iv)} we used, as we will often do in the following, the notation $x^+$ as a short-hand for $\max\{0,x\}$. We can now give a proof of the assertion that the asymptotic percolation rate almost surely equals a non\hyp{}random number.
\begin{proof}[Proof of \cref{prop-FPP-rate}]
We will first argue that the family $\{X_{m,n}=l_{m\to n}:0\leq m\leq n\}$ satisfies the conditions of \cref{liggett}. In \cref{subadd} we have seen that (i) holds. (ii) and (iii) follow from the translational invariance of the graph and the fact that $\{w(e),e\in E\}$ is i.i.d. The last condition, (iv), holds by assumption and in particular for the exponential distribution. We can thus conclude that $l_n/n\to\chi$ almost surely as $n\to\infty$. To see that $\chi$ is indeed constant, we enumerate the edges $\{e\in E\}$ in some way, say $e_1, e_2,\dots$. Since $\chi$ is defined as $\lim_{n\to\infty}{l_n/n}$ and each $l_n$ is a sum of only a finite number of edge\hyp{}weights, namely of a subset of those contained in $G_n^{\mathcal{E}}$, $\chi$ is measurable with respect to $\bigcap_{n\geq 1}{\sigma\left(w(e_n),w(e_{n+1}),\ldots\right)}$, the tail-$\sigma$-algebra of the i.i.d. sequence $w(e_1), w(e_2),\dots$. Kolmogorov's Zero\hyp{}One law (see for example \citet[Theorem 3.12]{breiman1968p}) then implies that for any Borel set $B$ the probability that $\chi$ takes a value in $B$ is either $0$ or $1$ which means that $\chi$ is almost surely equal to a deterministic number.
\end{proof}
In view of our next theorem, which provides an explicit formula for the percolation rate $\chi$, we could have done without \cref{prop-FPP-rate} and without invoking Kingman's general subadditivity result. We chose to include it, however, in order to distinguish this universal aspect of first\hyp{}passage percolation theory from properties specific to our model. The result is the following: 
\begin{theorem}
\label{FPP-rate-formula}
Let $\Lambda_n\coloneqq l_n-l_{n-1}$ and assume there exists an $S$\hyp{}valued	 ergodic Markov chain $(M_n)_{n\geq 1}$ and a measurable function $f:S\to\mathbb{R}$ satisfying $\Lambda_n=f(M_n)$ for all positive integers $n$. Then, almost surely,
\begin{equation}
\chi=\lim_{n\to\infty}{\frac{l_n}{n}}=\int_S{f(s)\pi(s)},
\end{equation}
where $\pi$ is the unique invariant distribution of $(M_n)_{n\geq 1}$. Put differently, $\chi=\mathbb{E}\left[\Lambda\right]$, where $\Lambda$ is the weak limit of the sequence $(\Lambda_n)_{n\geq 1}$, i.e. $\Lambda_n\convd\Lambda$.
\end{theorem}
\begin{remark}
It will be established in \cref{HMM} that there indeed exists a Markov chain $(M_n)_{n\geq 1}$ and a function $f$ satisfying the conditions of the theorem.
\end{remark}
\begin{proof}
As an instantaneous function of the ergodic Markov chain $(M_n)_{n\geq 1}$ the sequence $(\Lambda_n)_{n\geq 1}$ is ergodic as well. It follows that the time average $\frac{1}{n}\sum_{\nu=1}^n\Lambda_\nu$ converges to $\mathbb{E}\left[\Lambda\right]$. Since clearly $l_n=\sum_{\nu=1}^n\Lambda_n$ the claim follows.
\end{proof}

\section{Calculations for \texorpdfstring{$\mathcal{G}^{\{\mathcal{X},\mathcal{Y},\mathcal{Z}\}}$}{the ladder case}}
\label{section-FPP}
In this section we prove our results about the time constants for first\hyp{}passage percolation on width\hyp{}two stretches (\cref{tabl:graphs-2}) in the case of $\mathcal{E}=\{\mathcal{X},\mathcal{Y},\mathcal{Z}\}$. The calculations for the other two cases mentioned are completely analogous and can be found in \citet{schlemm2008fpp}. We denote the length of the shortest path from $(0,0)$ to $(n,0)$ by $l_n$, the length of the shortest path from $(0,0)$ to $(n,1)$ by $l_n'$ and we let $\Delta_n=l_n'-l_n$. Our first goal is to find a recurrence relation between the distributions of $\Delta_n$ and $\Delta_{n-1}$.
\begin{proposition}
\label{TransKern-1}
The sequence of random variables $(\Delta_n)_{n\geq 0}$ is a real\hyp{}valued Markov chain with initial distribution $\mathbb{P}\left(\Delta_0\leq d\right)=1-e^{-d}$ and transition kernel
\begin{equation}
\label{transkernelK}
K(\delta,d)= e^{-|d|}\begin{cases}
1                   & d< 0      \wedge \delta \leq d,\\
e^{-(\delta-d)}  & d< 0      \wedge \delta> d,\\
e^{-|\delta|}    & d=0,\\
e^{-(d-\delta)}  & d>0       \wedge \delta\leq d,\\
1                   & d>0       \wedge \delta>d.
\end{cases}
\end{equation}
\end{proposition}
\begin{proof}
It is clear that we must try to express $\rho_n$, the probability density function of $\Delta_n$, in terms of $\rho_{n-1}$, the density function of $\Delta_{n-1}$. To achieve this we define events on which $\Delta_n$ is a deterministic function of $X_n$, $Y_n$, $Z_n$ and $\Delta_{n-1}$, which amounts to separately taking into account the different possible behaviours of the last step of the shortest path. Consider first the shortest path $p^{0\to 1}$ from $(0,0)$ to $(n,1)$, the last edge of which must be either $Y_n$ or $Z_n$. Similarly, the last edge of $p^{0\to 0}$, the shortest path from $(0,0)$ to $(n,0)$, must be either $X_n$ or $Z_n$. We note that almost surely, the edge $Z_n$ is not part of both $p^{0\to 1}$ and $p^{0\to 0}$ because this would contradict the uniqueness property mentioned in the introduction, so that we have to consider three cases:
\begin{subequations}
\label{events}
\begin{align}
\label{event1}\includegraphics[scale=1]{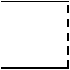}=&\left\{Y_n+l_{n-1}'\leq Z_n+X_n+l_{n-1}\wedge X_n+l_{n-1}\leq Z_n+Y_n+l_{n-1}'\right\}\\
=&\left\{Z_n\geq|X_n-Y_n-\Delta_{n-1}|\right\}\notag \\
\label{event2}\includegraphics[scale=1]{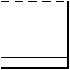}=&\left\{Z_n+X_n+l_{n-1}\leq Y_n+l_{n-1}'\wedge X_n+l_{n-1}\leq Z_n+Y_n+l_{n-1}'\right\}\\
=&\left\{Y_n\geq X_n+Z_n-\Delta_{n-1}\right\}\notag\\
\label{event3}\includegraphics[scale=1]{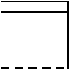}=&\left\{Y_n+l_{n-1}'\leq Z_n+X_n+l_{n-1}\wedge Z_n+Y_n+l_{n-1}'\leq X_n+l_{n-1}\right\}\\
=&\left\{X_n\geq Y_n+Z_n+\Delta_{n-1}\right\}\notag
\end{align}
\end{subequations}
In particular we have the relation
\begin{equation}
\label{DeltaRec}
\Delta_n=\min\{\Delta_{n-1}+Y_n,X_n+Z_n\}-\min\{X_n,\Delta_{n-1}+Y_n+Z_n\}.
\end{equation}
The reason why the pictographs above have been chosen to represent the different events is the following: solid lines correspond to edges being used by the shortest path to either $(n,0)$ or $(n,1)$, double solid lines represent edges being part of both these paths, while dashed lines stand for unused edges. Using these events we now compute the cumulative distribution function of $\Delta_n$ as
\begin{equation}
\label{DeltaCDf-1}
\mathbb{P}\left(\Delta_n\leq d\right)=\mathbb{P}\left(\left\{\Delta_n\leq d\right\}\cap\includegraphics[scale=0.5]{quad1.pdf}\right)+\mathbb{P}\left(\left\{\Delta_n\leq d\right\}\cap\includegraphics[scale=0.5]{quad2.pdf}\right)+\mathbb{P}\left(\left\{\Delta_n\leq d\right\}\cap\includegraphics[scale=0.5]{quad3.pdf}\right).
\end{equation}
On the event \includegraphics[scale=0.5]{quad1.pdf}, $\Delta_n$ is equal to $\Delta_{n-1}+Y_n-X_n$, so the first term is
\begin{equation*}
\mathbb{P}\left(\Delta_n\leq d\cap \includegraphics[scale=0.5]{quad1.pdf}\right)= \int_\mathbb{R}{d\delta\,\rho_{n-1}(\delta)\int_{\mathbb{R}^3}{d\mathbb{P}^3(x,y,z)\,\mathbb{I}_{\{z\geq|x-y-\delta|\}}\mathbb{I}_{\{\delta+y-x\leq d\}}}}.
\end{equation*}
Rewriting the indicator functions as integration bounds we obtain
\begin{equation*}
\int_{\mathbb{R}}{d\delta\,\rho_{n-1}(\delta)\int_{0}^{\infty}{d\mathbb{P}(y)\int_{(\delta+y-d)^+}^{\infty}{d\mathbb{P}(x)\int_{|x-y-\delta|}^{\infty}{d\mathbb{P}(z)}}}}
\end{equation*}
and after doing the tedious but easy $x$-, $y$- and $z$-integrals we arrive at
\begin{equation*}
\mathbb{P}\left(\Delta_n\leq d\cap \includegraphics[scale=0.5]{quad1.pdf}\right)=\int_\mathbb{R}{d\delta\,\rho_{n-1}(\delta)G_1(\delta,d)}
\end{equation*}
with
\begin{equation*}
G_1(\delta,d)=\frac{1}{4}\begin{cases}
e^{\delta}(1+2(d-\delta))                     & d<0           \wedge \delta\leq d,\\
e^{2(d-\delta)}                                  & d<0           \wedge \delta> d,\\
e^{\delta}(2-2\delta-e^{-2d})              & d\geq0        \wedge \delta\leq 0,\\
e^{-\delta}(2+2\delta-e^{-2(d-\delta)})    & d\geq0        \wedge 0<\delta\leq d,\\
e^{-\delta}(1+2\delta)                        & d\geq0        \wedge\delta> d.
\end{cases}
\end{equation*}
One can also confirm this computation and similar ones which follow by use of a computer algebra system. In fact, we used  Mathematica to verify our results. On the event \includegraphics[scale=0.5]{quad2.pdf}, $\Delta_n$ is given by $Z_n$, so for the second term in \eqref{DeltaCDf-1} we obtain:
\begin{align*}
\mathbb{P}\left(\Delta_n\leq d\cap \includegraphics[scale=0.5]{quad2.pdf}\right) =& \int_\mathbb{R}{d\delta\rho_{n-1}(\delta)\int_{\mathbb{R}^3}{d\mathbb{P}^3(x,y,z)\,\mathbb{I}_{\{y\geq x+z-\delta\}}\mathbb{I}_{\{z\leq d\}}}} \\						 =&\int_{\mathbb{R}}{d\delta\,\rho_{n-1}(\delta)\int_{0}^{\infty}{d\mathbb{P}(x)\int_0^{d^+}}{d\mathbb{P}(z)\int_{(x+z-\delta)^+}^{\infty}{d\mathbb{P}(y)}}}\\
\eqqcolon&\int_\mathbb{R}{d\delta\,\rho_{n-1}(\delta)G_2(\delta,d)},
\end{align*}
where the function $G_2$ is given by
\begin{equation*}
G_2(\delta,d)=\frac{1}{4}\mathbb{I}_{\{d\geq0\}}\begin{cases}
e^{\delta}(1-e^{-2d})                                     & \delta\leq 0,\\
4-e^{-\delta}\left(3+e^{-2(d-\delta)}+2\delta\right)   & 0<\delta\leq d,\\
4-4e^{- d}-2 d e^{-\delta}                             & \delta> d.
\end{cases}
\end{equation*}
Finally, on the event \includegraphics[scale=0.5]{quad3.pdf}, $\Delta_n$ is equal to $-Z_n$, so we compute for the third term in \eqref{DeltaCDf-1}
\begin{align*}
\mathbb{P}\left(\Delta_n\leq d\cap \includegraphics[scale=0.5]{quad3.pdf}\right) =& \int_\mathbb{R}{d\delta\rho_{n-1}(\delta)\int_{\mathbb{R}^3}{d\mathbb{P}^3(x,y,z)\,\mathbb{I}_{\{x\geq y+z+\delta\}}\mathbb{I}_{\{z\leq -d\}}}} \\ =&\int_{\mathbb{R}}{d\delta\,\rho_{-1}n(\delta)\int_{0}^{\infty}{d\mathbb{P}(y)\int_{(-d)+}^\infty{d\mathbb{P}(z)\int_{(y+z+\delta)^+}^{\infty}{d\mathbb{P}(x)}}}}\\
\eqqcolon&\int_\mathbb{R}{d\delta\,\rho_{n-1}(\delta)G_3(\delta,d)}
\end{align*}
and
\begin{equation*}
G_3(\delta,d)=\frac{1}{4}\begin{cases}
4e^{\delta}+e^{\delta}(-2(d-\delta)-3)     & d<0           \wedge \delta \leq d,\\
e^{-2(\delta-d)}                                 & d<0           \wedge \delta> d,\\
4+e^{\delta}(2\delta-3)                       & d\geq0        \wedge \delta\leq 0,\\
e^{-\delta}                                      & d\geq0        \wedge \delta>0.
\end{cases}
\end{equation*}
It is interesting to note that, due to the symmetry of the events \eqref{event2} and \eqref{event3}, the sum of $G_2(\delta,d)$ and $G_3(-\delta,-d)$ does not depend on $d$; explicitly it holds that
\begin{equation*}
G_2(\delta,d)+G_3(-\delta,-d)=\frac{1}{4}e^{-|\delta|}\begin{cases}
                                      1 & \delta \leq 0\\
                                      4 e^\delta-2\delta-3 & \delta>0
				      \end{cases}.
\end{equation*}
Putting everything together it follows that
\begin{equation*}
\mathbb{P}\left(\Delta_n\leq d\right)=\int_{\mathbb{R}}{d\delta\,\rho_{n-1}(\delta)\left[G_1(\delta,d)+G_2(\delta,d)+G_3(\delta,d)\right]}.
\end{equation*}
Differentiating this equation with respect to $d$ and interchanging differentiation and integration on the right hand side we obtain
\begin{equation*}
\rho_n(\delta)=\int_{\mathbb{R}}{d\delta\,\rho_{n-1}(\delta)\partial_d\left[G_1[\delta,d]+G_2(\delta,d)+G_3(\delta,d)\right]},
\end{equation*}
so the transition kernel $K:\mathbb{R}\to\mathbb{R}$ is given by $K(\delta,d)=\partial_d\left[G_1[\delta,d]+G_2(\delta,d)+G_3(\delta,d)\right]$ and the claim follows from basic computations.
\end{proof}
From the explicit description of the transition kernel $K(\delta,d)$ we can deduce useful properties of $(\Delta_n)_{n\geq 0}$:
\begin{lemma}
\label{DeltaConv}
The Markov chain $(\Delta_n)_{n\geq 0}$ is ergodic. In particular, as $n\to\infty$, $(\Delta_n)_{n\geq 0}$ converges in distribution to a non\hyp{}degenerate limiting random variable $\Delta$ with probability density function $\rho_\infty=\lim_{n\to\infty}{\rho_n}$.
\end{lemma}
\label{DeltaDistConv}
\begin{proof}
It is enough to note that the kernel function $K$ from \cref{TransKern-1} is strictly positive and continuous. The claim therefore follows from an extension of the Perron-Frobenius theorem to continuous transfer operators \citep[see][]{jentzsch1912uip}.
\end{proof}
\begin{corollary}
\label{StatDistIntEq-1}
Let $\rho_\infty=\lim_{n\to\infty}{\rho_n}$. Then
\begin{enumerate}[(i)]
\item The density $\rho_\infty$ of the stationary distribution of $(\Delta_n)_{n\geq 0}$ satisfies the integral equation
\begin{subequations}
\label{inteq}
\begin{align}
\label{inteqpos}\rho_\infty(d)=& e^{ d}\int_{-\infty}^d{d\delta\,\rho_\infty(\delta)}+ e^{2 d}\int_{d}^{\infty}{d\delta\,\rho_\infty(\delta)e^{-\delta}} & d<0\\
\label{inteqneg}\rho_\infty(d)=&e^{-2 d}\int_{-\infty}^d{d\delta\,\rho_\infty(\delta)e^{\delta}}+ e^{- d}\int_{d}^{\infty}{d\delta\,\rho_\infty(\delta)} & d\geq0.
\end{align}
\end{subequations}
\item The density $\rho_\infty$ is an even function, that is $\rho_\infty(d)=\rho_\infty(-d)$ for all $d\in\mathbb{R}$.
\end{enumerate}
\end{corollary}
\begin{proof}
For the first claim we observe that $\rho_\infty$ satisfies the integral equation $\rho_\infty(d)=\int_{\mathbb{R}}{d\delta\,\rho_\infty(\delta)K(\delta,d)}$ with the kernel $K$ given in \eqref{transkernelK}. The second assertion follows from the symmetry $K(\delta,d)=K(-\delta,-d)$.
\end{proof}
In order to solve this integral equation we transform it into a differential equation.
\begin{lemma}
\label{StatDistODE-1}
The density $\rho_\infty$ of the stationary distribution of $(\Delta_n)_{n\geq 0}$ satisfies the differential equation
\begin{subequations}
\label{eq:StatDistODE-1}
\begin{align}
\rho_\infty''(d) =& -\left[2+e^{ d}\right]\rho_\infty(d)+3\rho_\infty'(d) &d<0\label{eq:StatDistODE-1-neg},\\
\rho_\infty''(d) =& -\left[2+e^{- d}\right]\rho_\infty(d)-3\rho_\infty'(d) &d\geq0\label{eq:StatDistODE-1-pos}.
\end{align}
\end{subequations}
\end{lemma}
\begin{proof}
We only prove the claim for $d<0$. The case $d\geq0$ can be shown in the same way or one uses (ii) of \cref{StatDistIntEq-1}. In the case $d<0$ we know that $\rho_\infty$ is a solution to the integral equation \eqref{inteqpos}.
Differentiating this equation with respect to $d$ and using \eqref{inteqpos} again we obtain
\begin{align}
\rho_\infty'(d)=& e^{ d}\rho_\infty(d)+e^{d}\int_{-\infty}^d{d\delta\,\rho_\infty(\delta)}\notag\\
            \label{difffirststep}&- e^{d}\rho_\infty(d)+2 e^{2 d}\int_{d}^{\infty}{d\delta\,\rho_\infty(\delta)e^{-\delta}}\\
            =&2\rho_\infty(d)-e^{ d}\int_{-\infty}^d{d\delta\,\rho_\infty(\delta)}\notag.
\end{align}
Differentiating again and eliminating the remaining integral via \eqref{difffirststep} yields
\begin{align*}
\rho_\infty''(d)=& 2\rho_\infty'(d)-e^{ d}\rho_\infty(d)-e^{ d}\int_{-\infty}^d{d\delta\,\rho_\infty(\delta)}\\
                =&3\rho_\infty'(d)-\left[2+e^{ d}\right]\rho_\infty(d).
\end{align*}
\end{proof}
\begin{proposition}
\label{StatDist-1}
The density $\rho_\infty$ of the stationary distribution of the Markov chain $(\Delta_n)_{n\geq 0}$ is
\begin{equation}
\rho_\infty(d)=\frac{1}{2\operatorname{J}_2(2)}e^{-\frac{3}{2}|d|}\operatorname{J}_1\left(2e^{-\frac{1}{2}|d|}\right),
\end{equation}
where $\operatorname{J}_\nu$ are Bessel functions of the first kind.
\end{proposition}
\begin{proof}
For $d<0$, we write $\rho_\infty$ as $\rho_\infty(d)=e^{\frac{3}{2} d}\tilde \rho_\infty\left(2e^{\frac{1}{2} d}\right)$ with a function $\tilde \rho_\infty$ which is to be determined. If this ansatz is inserted into \eqref{eq:StatDistODE-1-neg} and if we replace $d$ by $2\log\frac{z}{2}$ we obtain $z^2 \frac{d^2\tilde\rho_\infty}{dz^2}(z)+z\frac{d\tilde\rho_\infty}{dz}(z)+(z^2-1)\tilde\rho_\infty(z)=0$. This is the differential equation the solution to which are by definition the Bessel functions of the first and second kind, denoted by $\operatorname{J}_\nu$ and $\operatorname{Y}_\nu$, respectively \citep[see][equation 9.1.1]{abramowitz1965hmf}. In this particular case $\nu=1$. One can apply an analogous method for the case $d>0$ to find that the general solution of equation \eqref{eq:StatDistODE-1} is given by
\begin{equation}
\label{solutiondens}
\rho_\infty(d)=e^{-\frac{3}{2} |d|}\left[c_1\operatorname{J}_1\left(2e^{-\frac{1}{2} |d|}\right)+c_2 \operatorname{Y}_1\left(2e^{-\frac{1}{2} |d|}\right)\right].
\end{equation}
In order to determine the two constants $c_1$ and $c_2$ we insert this general solution into the integral equation \eqref{inteq} and take the limit $d\to 0$ (it does not make a difference whether we take $d<0$ or $d\geq0$). We choose $d<0$ and obtain
\begin{equation*}
\rho_\infty(0) = \int_{-\infty}^0{d\delta \rho_\infty(\delta)}+\int_0^\infty{d\delta \rho_\infty(\delta)e^{-\delta}},
\end{equation*}
in which we replace $\rho_\infty(\delta)$ by the expression given in \eqref{solutiondens} to get
\begin{align*}
c_1\operatorname{J}_1\left(2\right)+c_2 \operatorname{Y}_1\left(2\right) =& c_1\left[\int_{-\infty}^0{d\delta\, e^{\frac{3}{2} d}\operatorname{J}_1\left(2e^{\frac{1}{2} d}\right)} + \int_0^\infty{d\delta\, e^{-\frac{3}{2} d}\operatorname{J}_1\left(2e^{-\frac{1}{2} d}\right)e^{-\delta}}\right]\\
&+c_2\left[\int_{-\infty}^0{d\delta\,e^{\frac{3}{2} d}\operatorname{Y}_1\left(2e^{\frac{1}{2} d}\right)} + \int_0^\infty{d\delta \,e^{-\frac{3}{2} d}\operatorname{Y}_1\left(2e^{-\frac{1}{2} d}\right)e^{-\delta}}\right]\\
=&c_1\left[\operatorname{J}_2\left(2\right)+\operatorname{J}_0\left(2\right)\right]+c_2 \left[\left(\frac{1}{\pi} + \operatorname{Y}_2\left(2\right)\right) + \left(-\frac{2}{\pi} + \operatorname{Y}_0\left(2\right)\right)\right],
\end{align*}
Since $\operatorname{J}_1\left(2\right)=\operatorname{J}_0\left(2\right)+\operatorname{J}_2\left(2\right)$, $\operatorname{Y}_1\left(2\right)=\operatorname{Y}_0\left(2\right)+\operatorname{Y}_2\left(2\right)$ and $-1/\pi\neq 0$ it follows that $c_2=0$. The requirement that $\rho_\infty$ integrate to one implies
\begin{align*}
c_1^{-1} =& \int_{\mathbb{R}}{d\delta\,e^{-\frac{3}{2} |\delta|}\operatorname{J}_1\left(2e^{-\frac{1}{2} |\delta|}\right)}=\frac{1}{2}\int_0^2{dy\,y^2 \operatorname{J}_1(y)}=2\operatorname{J}_2\left(2\right).
\end{align*}
\end{proof}
\begin{lemma}
\label{HMM}
There exists an $\mathbb{R}\times \mathbb{R}_+^3$\hyp{}valued ergodic Markov chain $(M_n)_{n\geq 1}$ and a function $f:\mathbb{R}\times \mathbb{R}_+^3\to\mathbb{R}$ such that $\Lambda_n=f(M_n)$ for all $n\geq 1$.
\end{lemma}
\begin{proof}
Set $M_n=(\Delta_{n-1},X_n,Y_n,Z_n)$ and $f(\delta,x,y,z)=\min\{x,\delta+y+z\}$. It is then clear that $f(M_n)=\Lambda_n$ (see \eqref{LambdaRec}) and $(M_n)_{n\geq 1}$ is Markov because $\Delta_{n-1}$ can be written as $\min\{\Delta_{n-2}+Y_{n-1},X_{n-1}+Z_{n-1}\}-\min\{X_{n-1},\Delta_{n-2}+Y_{n-1}+Z_{n-1}\}$. (cf. equation \eqref{DeltaRec} in the proof of \cref{TransKern-1}.) Ergodicity is a direct consequence of \cref{DeltaConv}.
\end{proof}

We will now use this result to compute the distribution of $\Lambda\coloneqq\lim_{n\to\infty}{\Lambda_n}$, the expected value of which is, by \cref{FPP-rate-formula}, the percolation rate we are looking for.
\begin{lemma}
\label{StatDistLambda-1}
For each natural number $n$, the density $\eta_n$ of the distribution of $\Lambda_n$ is $
\eta_n(l)=\int_{\mathbb{R}}{d\delta\, \rho_{n-1}(\delta)Q(\delta,l)}$ with $Q:\mathbb{R}^2\to\mathbb{R}$ given by
\begin{equation}
Q(\delta,l)= e^{- l}\begin{cases}
 e^{\delta}(l-\delta)                         & l<0           \wedge \delta\leq l,\\
e^{-(l-\delta)}\left(1+2(l-\delta)\right) & l\geq0        \wedge \delta\leq l,\\
1                                               & l\geq0        \wedge \delta >l,\\
0                                               & \text{otherwise}.
\end{cases}
\end{equation}
In particular, $\eta_\infty(l)=\int_{\mathbb{R}}{d\delta\, \rho_\infty(\delta)Q(\delta,l)}$.
\end{lemma}
\begin{proof}
As before we define events for each possible behaviour of the last step of the shortest path from $(0,0)$ to $(0,n)$:
\begin{align*}
\includegraphics[scale=1]{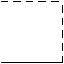}=&\left\{X_n+l_{n-1}\leq Z_n+Y_n+l_{n-1}'\right\}=\left\{X_n\leq \Delta_{n-1}+Y_n+Z_n\right\} \\
\includegraphics[scale=1]{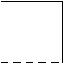}=&\left\{X_n+l_{n-1}\geq Z_n+Y_n+l_{n-1}'\right\}=\left\{X_n\geq \Delta_{n-1}+Y_n+Z_n\right\}.
\end{align*}
Clearly,
\begin{equation}
\label{LambdaRec}
\Lambda_n=\min\{X_n,\Delta_{n-1}+Y_n+Z_n\},
\end{equation}
i.e. on the first of the two events, $\Lambda_n$ is given by $X_n$ while on the second it equals $Y_n+Z_n+\Delta_{n-1}$. Now the procedure continues very similar to the proof of \cref{TransKern-1}. We compute the cumulative distribution function of $\Lambda$ as
\begin{equation}
\label{LambdaCDF-1}
\mathbb{P}\left(\Lambda_n\leq l\right)=\mathbb{P}\left(\{\Lambda_n\leq l\}\cap \includegraphics[scale=0.5]{quad4.pdf}\right) +\mathbb{P}\left(\{\Lambda_n\leq l\}\cap\includegraphics[scale=0.5]{quad5.pdf}\right),
\end{equation}
because on each of these two events we have an explicit expression for $\Lambda_n$. The first term leads to
\begin{align*}
\mathbb{P}\left(\{\Lambda_n\leq l\}\cap \includegraphics[scale=0.5]{quad4.pdf}\right)=& \int_\mathbb{R}{d\delta\,\rho_{n-1}(\delta)\int_{\mathbb{R}^3}{d\mathbb{P}^3(x,y,z)\,\mathbb{I}_{\{x\leq\delta+y+z\}}\mathbb{I}_{\{x\leq l\}}}}\\
=&\int_{\mathbb{R}}{d\delta\,\rho_{n-1}(\delta)\int_{0}^{\infty}{d\mathbb{P}(z)\int_0^{\infty}{d\mathbb{P}(y)\int_0^{\min{\{\delta+y+z,l\}}^+}{d\mathbb{P}(x)}}}}  \\
\eqqcolon& \int_{\mathbb{R}}{d\delta\,\rho_{n-1}(\delta)P_1(\delta,l)},
\end{align*}
where the function $P_1$ is the result of the $x$-, $y$- and $z$-integration and given explicitly by
\begin{equation*}
P_1(\delta,l)=\mathbb{I}_{\{l\geq0\}}\begin{cases}
\frac{1}{4}e^{-2 l+\delta}\left[e^{2 l}(3-2\delta)-3-2(l-\delta)\right]   & \delta\leq0,\\
\frac{1}{4}\left[4-e^{-\delta}-e^{-2 l+\delta}(2(l-\delta)+3)\right]        & 0<\delta\leq l,\\
1-e^{- l}                                                                                & \delta>l.
\end{cases}
\end{equation*}
Very similarly we obtain for the second term in \eqref{LambdaCDF-1}
\begin{align*}
\mathbb{P}\left(\{\Lambda_n\leq l\}\cap \includegraphics[scale=0.5]{quad4.pdf}\right)=& \int_\mathbb{R}{d\delta\,\rho_{n-1}(\delta)\int_{\mathbb{R}^3}{d\mathbb{P}^3(x,y,z)\,\mathbb{I}_{\{x\geq\delta+y+z\}}\mathbb{I}_{\{\delta+y+z\leq l\}}}}\\
=&\int_{\mathbb{R}}{d\delta\,\rho_{n-1}(\delta)\int_{0}^{\infty}{d\mathbb{P}(z)\int_0^{(l-z-\delta)^+}{d\mathbb{P}(y)\int_{(\delta+y+z)^+}^{\infty}{d\mathbb{P}(x)}}}}\\
\eqqcolon&\int_{\mathbb{R}}{d\delta\,\rho_{n-1}(\delta)P_2(\delta,l)},
\end{align*}
with $P_2$ given by
\begin{equation*}
P_2(\delta,l)=\begin{cases}
1-e^{ l-\delta}(1-\delta+ l)                                                    & l<0   \wedge \delta\leq l,\\
1-\frac{1}{4}e^{\delta}\left[3-2\delta+e^{-2 l}(1+2(l-\delta))\right]           & l\geq0\wedge\delta\leq0,\\
\frac{1}{4}e^{-2 l-\delta}\left[e^{2 l}-e^{2\delta}(1+2(l-\delta))\right]    & l\geq0\wedge 0\leq\delta\leq l,\\
0                                                                                           & \text{otherwise}.
\end{cases}
\end{equation*}
Putting things together, it follows that $\mathbb{P}\left(\Lambda_n\leq l\right)=\int_{\mathbb{R}}{d\delta\,\rho_{n_1}(\delta)\left[P_1(\delta,l)+P_2(\delta,l)\right]}$. Taking the derivative with respect to $l$ and interchanging differentiation and integration on the right hand side we obtain $Q(\delta,l)=\partial_l\left[P_1(\delta,l)+P_2(\delta,l)\right]$. The statement about the relation between the stationary densities $\eta_\infty$ and $\rho_\infty$ is clear.
\end{proof}
\begin{theorem}
\label{percolationrate1}
The percolation rate $\chi$ on $\mathcal{G}^{\{\mathcal{X},\mathcal{Y},\mathcal{Z}\}}$ is $\left[\frac{3}{2}-\frac{\operatorname{J}_1(2)}{2\operatorname{J}_2(2)}\right]
\approx 0.68\dots$.
\end{theorem}
\begin{proof}
We have already argued that the time constant is the expectation of $\Lambda$, that is $\chi=\int_\mathbb{R}{dl\,l\,\eta_\infty(l)}$. From \cref{StatDistLambda-1} we know the explicit form of $\eta_\infty$ so we obtain $\chi=\int_\mathbb{R}{d\delta\,\rho_\infty(\delta)\int_\mathbb{R}{dl\,lQ(\delta,l)}}$, where the change of the order of integration is easily justified using Fubini's theorem. A direct calculation shows
\begin{equation*}
\int_\mathbb{R}{dl\,lQ(\delta,l)}=\frac{1}{4}\begin{cases}
8+4\delta-e^{\delta}(5-2\delta)&  \delta <0 \\
4-e^{-\delta}& \delta\geq 0
\end{cases},
\end{equation*}
and the percolation rate is
\begin{align*}
\chi =&\frac{1}{2\operatorname{J}_2(2)}\left[ \int_{\mathbb{R}_-}{d\delta\,e^{\frac{3}{2}\delta}\operatorname{J}_1\left(2e^{\frac{1}{2}\delta}\right)\left[8+4\delta-e^{\delta}(5-2\delta)\right]}\right.\\
&+\left.\int_{\mathbb{R}_+}{d\delta\,e^{-\frac{3}{2}\delta d}\operatorname{J}_1\left(2e^{-\frac{1}{2}\delta}\right)\left[4-e^{-\delta}\right]}\right]\\
=&\frac{1}{2\operatorname{J}_2(2)}\left[\frac{4\operatorname{J}_1(2)-7\operatorname{J}_0(2)}{4}+\frac{4\operatorname{J}_2(2)-\operatorname{J}_0(2)}{4}\right]=\frac{3}{2}-\frac{\operatorname{J}_1(2)}{2\operatorname{J}_2(2)}.
\end{align*}
\end{proof}
\section{Discussion}
The differences between the omitted computations for the cases $\{\mathcal{V},\mathcal{W},\mathcal{X},\mathcal{Y}\}$ and $\{\mathcal{W},\mathcal{X},\mathcal{Y},\mathcal{Z}\}$ and the ones presented are only a matter of degree, not of kind; for instance, there are not only three different cases to consider when proving the analogues of \cref{TransKern-1}, but rather eight and six, respectively. Moreover, for the case $\{\mathcal{W},\mathcal{X},\mathcal{Y},\mathcal{Z}\}$ the stationary density $\rho_\infty$ of $(\Delta_n)_{n\geq 1}$ (cf.\ \cref{StatDist-1}) is not symmetric. It is also the determination of this stationary distribution where things get more difficult with the unsolved cases: the characterizing differential equation (cf.\ \cref{StatDistODE-1}) then becomes non\hyp{}local, more specifically it involves both $\rho_\infty(d)$ and $\rho_\infty(-d)$, as well as their derivatives,  at the same time and so one can not solve it separately for $d>0$ and $d<0$ as we did. It might be possible to remedy this by computing $\rho_\infty$ not as an eigen function to the transition kernel $K$ itself but rather as an eigen function to its second convolution power $K^{(2)}(\delta,d)=\int_\mathbb{R}{d\sigma}K(\delta,\sigma)K(\sigma,d)$. As the method used in this paper is very similar to that in its antecedent \citet{flaxman2006fpp} its range of applicability is also essentially the same. It would be a natural generalization to consider graphs with vertex sets $\{1,\ldots,n\} \times \{0,\ldots,k\}$ for some integer $k \geq 1$ and for the directed first\hyp{}passage percolation problem, similar techniques as used here apply equally well to this more general set\hyp{}up. However, the combinatorial difficulties arising from the need to explicitly keep track of the shortest paths on $G_{n+1} \backslash G_n$ seem very hard to overcome and for undirected percolation a similarly easy recursive argument as in the case $k = 2$ is not an option.\newline Many thanks go to B\'alint Vir\'ag for very helpful advice.

\end{document}